\documentclass[reqno, 11pt]{amsart}

\usepackage{amsthm}

\usepackage[dvipsnames]{xcolor}

\usepackage{footmisc}
\usepackage{tikz}
\usepackage{tikz-cd}
\usetikzlibrary{shapes,arrows,cd}
\usetikzlibrary{calc,arrows.meta,positioning}
\usetikzlibrary{decorations.pathreplacing}

\usepackage{mathrsfs,amssymb}

\usepackage{enumitem}

\usepackage{graphicx}
\usepackage{amsmath}
\usepackage{latexsym}

\usepackage[outline]{contour}
\contourlength{1.2pt}

\usepackage[
breaklinks=true,colorlinks=true,
linkcolor=black,urlcolor=black,citecolor=black,
bookmarks=true,bookmarksopenlevel=2]{hyperref}
\usepackage[capitalise, nameinlink, noabbrev]{cleveref}
\usepackage{nameref}

\usepackage[all]{xy}

\usepackage[numbers]{natbib}

\usepackage{url}
\usepackage{float}

\usepackage{lmodern}

\usepackage[textwidth=20mm]{todonotes}
\usepackage{setspace}

\newtheorem{theorem}{Theorem}[section]
\newtheorem{corollary}[theorem]{Corollary}
\newtheorem{proposition}[theorem]{Proposition}
\newtheorem{lemma}[theorem]{Lemma}

\theoremstyle{definition}
\newtheorem{definition}[theorem]{Definition}
\newtheorem{remark}[theorem]{Remark}
\newtheorem{question}[theorem]{Question}
\newtheorem{example}[theorem]{Example}

\theoremstyle{plain}
\newtheorem{claim}{Claim}

\newtheoremstyle{named}{}{}{\itshape}{}{\bfseries}{.}{.5em}{\thmnote{#3}}
\theoremstyle{named}
\newtheorem*{namedtheorem}{Theorem}

\DeclareMathOperator{\Inv}{Inv}
\DeclareMathOperator{\SGamma}{\textit{S} \, \Gamma}
\DeclareMathOperator{\Language}{\emph{L}}
\DeclareMathOperator{\Lin}{Lin} 

\begin{document}

\title[Subsemigroups of monogenic free inverse semigroup]{\normalsize Generators and presentations of inverse subsemigroups of the monogenic free inverse semigroup}


\noindent






\date{\today}


\maketitle

\begin{center}
{\normalsize
JUNG WON CHO%
\hspace{-.00em}\footnote{School of Mathematics and Statistics, University of St Andrews, St Andrews, Scotland,
UK. {\it Email: } {\tt jwc21@st-andrews.ac.uk}}
AND 
NIK RU\v{S}KUC
\hspace{-.20em}\footnote{School of Mathematics and Statistics, University of St Andrews, St Andrews, Scotland,
UK. {\it Email: }{\tt  nik.ruskuc@st-andrews.ac.uk. }{Supported by the Engineering and Physical Sciences Research Council [EP/V003224/1]}}
}
\end{center}

\begin{abstract}
It was proved by Oliveira and Silva (2005) that every finitely generated inverse subsemigroup of the monogenic free inverse semigroup $FI_1$ is finitely presented. 
The present paper continues this development, and gives generating sets and presentations for general (i.e.\ not necessarily finitely generated) inverse subsemigroups of $FI_1$. For an inverse semigroup $S$ and an inverse subsemigroup $T$ of $S$, we say $S$ is \textit{finitely generated modulo} $T$ if there is a finite set $A$ such that $S = \langle T, A \rangle$. Likewise, we say that $S$ is \textit{finitely presented modulo} 
$T $ if $S$ can be defined by a presentation of the form
$\Inv \langle X, Y \mid R, Q\rangle$, where $\Inv\langle X\mid R\rangle$ is a presentation for $T$ and $Y$ and $Q$ are finite. We show that every inverse subsemigroup $S$ of $FI_1$ is finitely generated modulo its semilattice of idempotents $E(S)$. 
By way of contrast, we show that when $S\neq E(S)$, it can never be finitely presented modulo $E(S)$. 
However, in the process we establish some nice (albeit infinite) presentations for $S$ modulo $E(S)$.
\end{abstract}

\section{Introduction}
\label{sec:intro}

Inverse semigroups form a variety of algebras, and therefore free inverse semigroups are guaranteed to exist \cite{schein_variety}. 
Among the free inverse semigroups, the monogenic one (denoted $FI_1$) has much simpler structure than the rest, and yet it is considerably more complicated
than its `plain' semigroup or group counterparts, namely the monogenic free  semigroup $\mathbb{N}$ or the cyclic free group $\mathbb{Z}$.
For example, both $\mathbb{Z}$ and $\mathbb{N}$ are commutative, 
whereas $FI_1$ is not. 
Furthermore,  every subgroup of $\mathbb{Z}$ is cyclic, every subsemigroup of $\mathbb{N}$ is finitely generated (see \cite{natural}), but $FI_1$ contains non-finitely generated inverse subsemigroups, e.g. its semilattice of idempotents.

The purpose of this paper is to consider finite generation and presentability properties of  arbitrary inverse subsemigroups of $FI_1$. The starting point is the following:

\begin{namedtheorem}[Theorem 1 \normalfont(\citet{Oliveira2005})]
\label{thm:o&s}
Every finitely generated inverse subsemigroup of the monogenic free inverse semigroup is finitely presented as an inverse semigroup.
\end{namedtheorem}

As the key stepping stone towards our main results we prove the following:

\begin{namedtheorem}[\nameref{thm:new_S_bar}]
If $S$ is an inverse subsemigroup of the monogenic free inverse semigroup generated by non-idempotent elements, then $S$ is finitely generated, and hence finitely presented.
\end{namedtheorem}

For an arbitrary inverse subsemigroup of $FI_1$ we explore to what extent is its non-finite generation or presentability `caused' by its idempotents. To formalise this, we will use the following notions:

\begin{definition}
\label{def:gen_mod}
Let $S$ be an inverse semigroup, and let $T$ be an inverse subsemigroup of~$S$. 
We say that $S$ is 
\begin{itemize}[leftmargin=8mm]
\item
\textit{finitely generated modulo} $T$ if there is a finite set $A$ such that $S = \langle T, A \rangle$;
\item
\textit{finitely presented modulo} $T$ if $S$ can be defined by an inverse semigroup presentation of the form
$\Inv \langle X, Y \mid R, Q\rangle$, where $\Inv\langle X\mid R\rangle$ 
is a presentation for $T$, and $Y$ and $Q$ are finite sets.
\end{itemize}
\end{definition}

Finite generation of semigroups modulo subsemigroups has already been studied in literature, e.g.\ in \cite{ArMi07,HoRuHi98}.
Presentations for groups modulo subgroups were considered in~\cite{Pride}.

As an immediate corollary of \nameref{thm:new_S_bar} we obtain:

\begin{namedtheorem}[\nameref{cor:gen}]
Every inverse subsemigroup of the monogenic free inverse semigroup is finitely generated modulo the semilattice of its idempotents.
\end{namedtheorem}

Moving on to presentations, our motivating question is whether there is a `presentation version' of \nameref{cor:gen}. 
Perhaps slightly surprisingly, or disappointingly, the answer turns out to be emphatically negative:

\begin{namedtheorem}[\nameref{thm:new_no_fin}]
Let $S$ be a non-semilattice inverse subsemigroup of the monogenic free inverse semigroup. Then, $S$ is not finitely presented modulo the semilattice of its idempotents.  
\end{namedtheorem}

However, in proving this theorem, we are able to obtain two `nice', albeit infinite, presentations for an arbitrary subsemigroup $S$, one as an amalgamated free product (Theorem~\ref{thm:new_amlagam}) and the other in terms of conjugation action on the idempotents (Corollary~\ref{cor:new_conj}).

The paper is structured as follows. \cref{sec:prelim} introduces notation, definitions and results about inverse semigroups, the monogenic free inverse semigroup and Sch\"{u}tzenberger graphs which will be used throughout this paper. \cref{sec:gen} is concerned with generating sets of general inverse subsemigroups of the monogenic free inverse semigroup, and contains proofs of \nameref{thm:new_S_bar} and \nameref{cor:gen}. 
\cref{sec:new_present} is devoted to presentations,  including a proof of \nameref{thm:new_no_fin}. The paper concludes with a few remarks and further questions in \cref{sec:con}.

\section{Preliminaries}
\label{sec:prelim}
\subsection{Inverse semigroups}
For a detailed introduction to inverse semigroups we refer the reader to standard monographs such as \cite{Howie,Lawson,Petrich}. Let $S$ be an inverse semigroup. 
The set of idempotents of $S$ will be denoted by $E(S)$; by definition $E(S)$ is a semilattice. 
The \emph{natural partial order} on $S$ is given by:
\[
a \leq b \text{ if there exists an idempotent $e \in S$ such that $a = eb$}.
\]

Free inverse semigroups can be defined as quotient semigroups. Let $X$ be a set, and let 
$X^{-1} := \{x^{-1} \colon x \in X \}$ be a set disjoint from $X$. 
Let $(X \cup X^{-1})^{+}$ be the free semigroup on $X \cup X^{-1}$. We define formal inverses of elements of $(X \cup X^{-1})^+$ by
\[
\begin{aligned}
    &(x^{-1})^{-1} := x \; \text{ for $x \in X $},\\
    &(x_1x_2\dots x_n)^{-1} := x_n^{-1}\dots x_2^{-1}x_1^{-1} \; \text{ where $x_i \in X \cup X^{-1}$ for all $i = 1, \dots, n$}.
\end{aligned}
\]
Let $\tau$ be the congruence on $(X \cup X^{-1})^+$ generated by 
\[
\{(uu^{-1}u, u) \colon u \in (X\cup X^{-1})^{+} \} \cup \{(uu^{-1}vv^{-1}, vv^{-1}uu^{-1}) \colon u, v \in (X \cup X^{-1})^+ \}.
\]
The quotient semigroup $(X\cup X^{-1})^+/ \tau$ is the \textit{free inverse semigroup on $X$}. 
This paper only considers the monogenic free inverse semigroup, obtained when $X=\{x\}$ is a singleton;
we denote this semigroup by $FI_1$. 

In a concrete representation, the elements of $FI_1$ can be viewed as triples:
\[
FI_1 = \{(-a, p, b) \in \mathbb{N}_0\times\mathbb{Z}\times\mathbb{N}_0 \colon  a+b>0,\ -a\leq p \leq b \}.
\]
The free generator $x$ corresponds to $(0, 1, 1)$. 
A typical triple $(-a, p, b)$ corresponds to $x^{-a}x^ax^bx^{-b}x^{p}$. 
The multiplication and inversion in $FI_1$ are as follows:
\[
\begin{aligned}
& (-a_1, p_1, b_1)(-a_2, p_2, b_2) = (-\max(a_1, a_2-p_1), p_1+p_2, \max(b_1, b_2+p_1)),\\
& (-a,p,b)^{-1}= (-(a+p), -p, b-p).
\end{aligned}
\]
For a more detailed account see  \cite[Chapter IX]{Petrich}.

We assume the reader is familiar with Green's equivalences, particularly $\mathscr{R}$ and $\mathscr{D}$; for an introduction see \cite{Clifford_Preston,Howie}.
In $FI_1$ the equivalences $\mathscr{R}$ and $\mathscr{D}$ are given as follows
\[
\begin{aligned}
&(-a, p, b)\mathscr{R} (-a', p', b') \Leftrightarrow a=a' \text{ and } b=b',\\
&(-a, p, b)\mathscr{D} (-a', p', b') \Leftrightarrow a+b=a'+b'.
\end{aligned}
\]
We denote the $\mathscr{R}$- and $\mathscr{D}$-class of $(-a, p, b)$ by $R_{a,b}$ and $D_{a+b}$, respectively.
In particular, the $\mathscr{D}$-classes of $FI_1$ are indexed by natural numbers.

The idempotents of $FI_1$ are precisely the triples of the form $(-a,0,b)$. We will denote the semilattice of idempotents by $E(FI_1)$, or simply $E$.
The product (meet) of two idempotents is $(-a,0,b)(-c,0,d)=(-\max(a,c),0,\max(b,d))$.
As a partially ordered set, $E$ is isomorphic to $\mathbb{N}_0\times\mathbb{N}_0\setminus\{(0,0)\}$, and
we will visualise it as shown in
\cref{fig:new_diag_1}. 

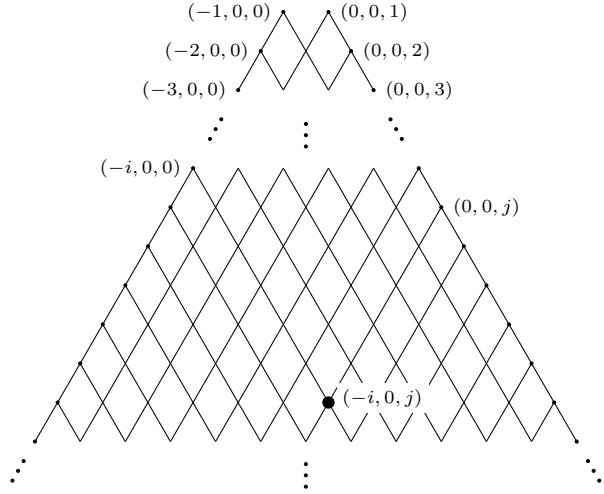
\begin{figure}
\begin{center}
\begin{tikzpicture}[scale=0.6]
\node [draw, shape = circle, fill = black, label=right:{\tiny $(0,0,1)$}, inner sep=-2, scale=0.2] (f1) at (0.5,8) {};
\node[draw, shape=circle, fill=black, label=left:{\tiny $(-1,0,0)$}, inner sep=-2, scale=0.2] (e1) at (-0.5,8) {}; 
\foreach \i in {2, ..., 3}{
    \node[draw, shape=circle, fill=black, label=right:{\tiny $(0,0,\i)$}, inner sep=-2, scale=0.2] (f\i) at ($(0.5,8)+(-60:\i-1)$) {};
    \node[draw, shape=circle, fill=black, label=left:{\tiny $(-\i, 0, 0)$}, inner sep=-2, scale=0.2] (e\i) at ($(-0.5,8)+(-120:\i-1)$) {};
}

\draw (f1)--(f3);
\draw (e1)--(e3);

\foreach \i in {1, ..., 3} {
    \draw ($(f1)+(-60:\i-1)$) -- ($(e3)+(0:\i)$);
    \draw ($(e1)+(-120:\i-1)$) -- ($(f3)-(0:\i)$);
}

\node [draw, shape = circle, fill = black, inner sep=-2, scale=0.2] at ($(e3)+(-120:0.75)$) {};
\node [draw, shape = circle, fill = black, inner sep=-2, scale=0.2] at ($(e3)+(-120:1)$) {};
\node [draw, shape = circle, fill = black, inner sep=-2, scale=0.2] at ($(e3)+(-120:1.25)$) {};

\node [draw, shape = circle, fill = black, inner sep=-2, scale=0.2] at ($(e3)!0.5!(f3) + (-90:0.75)$) {};
\node [draw, shape = circle, fill = black, inner sep=-2, scale=0.2] at ($(e3)!0.5!(f3) + (-90:1)$) {};
\node [draw, shape = circle, fill = black, inner sep=-2, scale=0.2] at ($(e3)!0.5!(f3) + (-90:1.25)$) {};

\node [draw, shape = circle, fill = black, inner sep=-2, scale=0.2] at ($(f3)+(-60:0.75)$) {};
\node [draw, shape = circle, fill = black, inner sep=-2, scale=0.2] at ($(f3)+(-60:1)$) {};
\node [draw, shape = circle, fill = black, inner sep=-2, scale=0.2] at ($(f3)+(-60:1.25)$) {};

\node [draw, shape = circle, fill = black, label = left:{\tiny $(-i,0,0)$}, inner sep=-2, scale=0.2] (ei1) at ($(e1)+(-120:4)$) {};
\node [draw, shape = circle, fill=black, inner sep=-2, scale = 0.2] (ei2) at ($(ei1)+(-120:1)$) {};
\node [draw, shape = circle, fill=black, inner sep=-2, scale = 0.2] (ei3) at ($(ei2)+(-120:1)$) {};
\node [draw, shape = circle, fill=black, inner sep=-2, scale = 0.2] (ei4) at ($(ei3)+(-120:1)$) {};
\node [draw, shape = circle, fill=black, inner sep=-2, scale = 0.2] (ei5) at ($(ei4)+(-120:1)$) {};
\node [draw, shape = circle, fill=black, inner sep=-2, scale = 0.2] (ei6) at ($(ei5)+(-120:1)$) {};
\node [draw, shape = circle, fill=black, inner sep=-2, scale = 0.2] (ei7) at ($(ei6)+(-120:1)$) {};
\node [draw, shape = circle, fill=black, inner sep=-2, scale=0.2] (ei8) at ($(ei7)+(-120:1)$) {};

\draw (ei1)--(ei8);

\node [draw, shape = circle, fill = black, inner sep=-2, scale=0.2] (fi1) at ($(f1)+(-60:4)$) {};
\node [draw, shape = circle, fill = black, label=right:{\tiny $(0,0,j)$}, inner sep=-2, scale=0.2] (fi2) at ($(fi1)+(-60:1)$) {};
\node [draw, shape = circle, fill = black, inner sep=-2, scale=0.2] (fi3) at ($(fi2)+(-60:1)$) {};

\foreach \i in {4, ...,8} {
    \node[draw, shape = circle, fill=black, inner sep=-2, scale = 0.2] (fi\i) at ($(fi1)+(-60:\i-1)$) {};
}
\draw (fi1)--(fi8);

\draw (ei1) -- ($(ei8)+(0:7)$);
\draw (ei2) -- ($(ei8)+(0:6)$);
\draw (ei3) -- ($(ei8)+(0:5)$);
\draw (ei4) -- ($(ei8)+(0:4)$);
\draw (ei5) -- ($(ei8)+(0:3)$);
\draw (ei6) -- ($(ei8)+(0:2)$);
\draw (ei7) -- ($(ei8)+(0:1)$);

\draw (fi1) -- ($(fi1)+(-120:7)$);
\draw (fi2) -- ($(fi2)+(-120:6)$);
\draw (fi3) -- ($(fi3)+(-120:5)$);
\draw (fi4) -- ($(fi4)+(-120:4)$);
\draw (fi5) -- ($(fi5)+(-120:3)$);
\draw (fi6) -- ($(fi6)+(-120:2)$);
\draw (fi7) -- ($(fi7)+(-120:1)$);

\foreach \i in {1, ...,4} {
    \draw ($(ei1)+(0:\i)$) -- ($(ei1)+(0:\i)+(-120:7)$);
    \draw ($(fi1)-(0:\i)$) -- ($(fi1)-(0:\i)+(-60:7)$);
}
\node [draw, shape = circle, fill = black, inner sep=-2.5, scale = 0.6] (eifj) at ($(ei1)+(-60:6)$) {};
\node [black, fill=white, inner sep=2pt] at ($(eifj) + (12mm,1mm)$) {\tiny $(-i,0,j)$};

\node [draw, shape = circle, fill = black, inner sep=-2, scale=0.2] at ($(ei8)+(-120:0.5)$) {};
\node [draw, shape = circle, fill = black, inner sep=-2, scale=0.2] at ($(ei8)+(-120:0.75)$) {};
\node [draw, shape = circle, fill = black, inner sep=-2, scale=0.2] at ($(ei8)+(-120:1)$) {};

\node [draw, shape = circle, fill = black, inner sep=-2, scale=0.2] at ($(ei8)!0.5!(fi8) + (-90:0.5)$) {};
\node [draw, shape = circle, fill = black, inner sep=-2, scale=0.2] at ($(ei8)!0.5!(fi8) + (-90:0.75)$) {};
\node [draw, shape = circle, fill = black, inner sep=-2, scale=0.2] at ($(ei8)!0.5!(fi8) + (-90:1)$) {};

\node [draw, shape = circle, fill = black, inner sep=-2, scale=0.2] at ($(fi8)+(-60:0.5)$) {};
\node [draw, shape = circle, fill = black, inner sep=-2, scale=0.2] at ($(fi8)+(-60:0.75)$) {};
\node [draw, shape = circle, fill = black, inner sep=-2, scale=0.2] at ($(fi8)+(-60:1)$) {};

\end{tikzpicture}
\\
\caption{Semilattice of idempotents of $FI_1$.}
\label{fig:new_diag_1}
\end{center}

\end{figure}

In this paper, we will use the term \emph{subsemigroup} to mean \emph{inverse subsemigroup}, i.e. a subset which is closed under both products and inversion, unless we explicitly say otherwise.
Thus, for an inverse semigroup $S$ and a set $X\subseteq S$, we write $\langle X\rangle$ to denote the inverse subsemigroup of $S$ generated by $X$.

The \textit{inverse semigroup presentation} $\Inv \langle X \mid R \rangle$
consists of a set  $X$, and a set of pairs of words $R\subseteq (X\cup X^{-1})^+\times (X\cup X^{-1})^+$.
It defines the inverse semigroup $S:=(X \cup X^{-1})^+/ \sigma$ where $\sigma$ is the congruence on $(X \cup X^{-1})^+$ generated by 
the congruence 
$\tau \cup R$ introduced earlier. 
For $w_1, w_2 \in (X\cup X^{-1})^+$ we will write $w_1 \equiv w_2$ to mean $w_1$ and $w_2$ are identical, and $w_1 = w_2$ to mean $w_1$ and $w_2$ are equal in $S$, i.e. $(w_1, w_2)\in \sigma$.

\subsection{Sch\"{u}tzenberger graphs}
By a \emph{graph} $\Gamma$ we will mean a directed labelled graph, with labels coming from $X\cup X^{-1}$ for some set $X$.
A \emph{birooted} graph is a triple $(\alpha, \Gamma, \beta)$, where $\Gamma$ is a graph, and
$\alpha$, $\beta$ are two chosen vertices. 
An edge from $u$ to $v$ with label $x$ will be written as $(u,x,v)$, and graphically represented as
$u \xrightarrow{\tiny\makebox[0.43cm]{$x$}} v$.

We say that $\Gamma$ is an \textit{inverse graph} if $u \xrightarrow{\tiny\makebox[0.43cm]{$x$}} v $ is an edge if and only if $u \xleftarrow{\scalebox{0.75}{$x^{-1}$}} v$ is an edge.
We then draw only one edge labelled by $x \in X \cup X^{-1}$ to represent both. 

An inverse graph $\Gamma$ is called \textit{deterministic} if it does not contain edges $(u, x, v)$ and $(u, x, v')$ with $v \neq v'$.

The \textit{language} of $(\alpha, \Gamma, \beta)$ is 
\[
\Language(\Gamma) := \{ w \in (X\cup X^{-1})^{\ast} \colon \text{$w$ labels a walk from $\alpha$ to $\beta$} \}
\]
where $(X \cup X^{-1})^{\ast}$ is the free monoid on $X \cup X^{-1}$.

Let $S$ be an inverse semigroup generated by $X$.
The \textit{Sch\"{u}tzenberger graph} $\SGamma(w)$ of $w$ with respect to $X$ is the birooted graph $(ww^{-1}, \Gamma, w)$, where $\Gamma$ is the subgraph of the right Cayley graph of $S$ with respect to $X$, induced on the vertex set $\textsf{V}:= R_{w}$, the $\mathscr{R}$-class of $w$. Thus, the edges of $\SGamma(w)$ are 
\begin{equation*}
\textsf{E}(\SGamma(w)) = \{(u, x, v)\in \textsf{V} \times (X \cup X^{-1}) \times \textsf{V} \colon ux = v \}
\end{equation*}
The language of $\SGamma(w)$ consists of the empty word and all $u \in (X\cup X^{-1})^{\ast}$ representing elements $\geq w$ in the natural partial order on $S$.

\begin{theorem}[{\cite[Theorems 3.1, 3.9]{Stephen1990}}]
\label{thm:stephen}
Let  $S$ and  $X$ be as above, and let $\sigma$ be the kernel of the natural epimorphism $(X\cup X^{-1})^{+}\rightarrow S$. For any  $w,v\in (X\cup X^{-1})^+$, we have:
\begin{enumerate}[leftmargin=8mm,itemsep=1mm,label=\textup{(\roman*)}]
    \item $\SGamma(w\sigma)$  is inverse and deterministic;
    \item $w\sigma=v\sigma$ if and only if $w \in \textit{L}(\SGamma(v\sigma))$ and $v \in \textit{L} (\SGamma(w\sigma))$.
\end{enumerate}

\end{theorem}

When $S$ is defined by a presentation $\Inv\langle X\mid R \rangle$, there is an iterative process for computing Sch\"{u}tzenberger graphs.
Let $w = x_1x_2\dots x_n \in (X \cup X^{-1})^+$. 
The \textit{linear graph} of $w$, $(\alpha_w, \Lin_w, \beta_w)$, is an inverse, birooted and directed graph with labels $x_1, x_2, \dots, x_n \in X\cup X^{-1}$. The vertices of $\Lin_w$ are $\alpha_w = \gamma_0, \gamma_1, \dots, \gamma_{n-1}, \gamma_n = \beta_w$ and the edges are $(\gamma_{i-1}, x_i, \gamma_i)$ for $i = 1, \dots, n$ as the following shows: 
\begin{figure}[H]\centering
\begin{tikzpicture}
\node (1) {$\gamma_0$};
\node (2) [right = 0.8cm of 1] {$\gamma_1$};
\node (3) [right = 0.8cm of 2] {$\gamma_2$};

\node (4) [right = 3.2cm of 3] {$\gamma_n$.};
\node (n-1) [left=0.8cm of 4] {$\gamma_{n-1}$};
\node at ($(3)!.3!(4)$) {\ldots};

\draw[->] (1) -- node[midway, above] {\small$x_1$} (2);
\draw[->] (2) -- node[midway, above] {\small$x_2$} (3);
\draw[->] (n-1) -- node[midway, above] {\small$x_n$} (4);

\end{tikzpicture}
\end{figure}
We define two operations on graphs called determinations and expansions.
\begin{description}[style=unboxed,leftmargin=0cm]
\item[Determination] If $\Gamma$ has  edges $(a, x, b)$ and $(a, x, c)$, then we identify two vertices $b$ with $c$ and identify the edge from $a$ to $b$ with the edge from $a$ to $c$. We call this operation a \textit{determination}. If $\Gamma'$ is a deterministic graph obtained from $\Gamma$ by a sequence of determinations, then we say that $\Gamma'$ is obtained from $\Gamma$ by a \textit{complete determination}. 
\item[Expansion] \begin{sloppypar}
Suppose $a, b$ are vertices of $\Gamma$ and $(r, s)$ or $(s, r)$ is in $R$. Suppose there is a walk from $a$ to $b$ with a label $r$, but no walk from $a$ to $b$ with a label $s = y_1\dots y_m \in (X \cup X^{-1})^+$. We add new vertices $c_1, \dots, c_{m-1}$ and new edges $(a, y_1, c_1), (c_1, y_2, c_2), \dots, (c_{m-1}, y_m,b)$. We call this operation an \textit{expansion}. If we perform an expansion for each $(r, s)$ or $(s, r)$ in $R$ which satisfies the assumption, then the resulting graph $\Gamma'$ is said to be obtained from $\Gamma$ by a \textit{complete expansion}.\end{sloppypar}
\end{description}

We define the \textit{Stephen's sequence} for $w \in (X\cup X^{-1})^+$ with respect to $\Inv\langle X \mid R \rangle$
is the sequence of graphs $\Gamma_k(w)$, $k\in\mathbb{N}_0$, starting at $(\alpha_w, \Lin_w, \beta_w)$, and then successively and alternately applying complete determinations and expansions.
Then, the Sch\"{u}tzenberger graph of $w$  is isomorphic to the limit of the sequence $\Gamma_k(w)$ (see \cite{Stephen1990}).

\section{Generating sets}
\label{sec:gen}

As we mentioned in the \nameref{sec:intro}, 
the monogenic free inverse semigroup $FI_1$ contains inverse subsemigroups that are not finitely generated. 
The most obvious example is the (infinite) semilattice of idempotents $E(FI_1)$. 
We also give a non-semilattice example:

\begin{example}
Let $S$ be the inverse subsemigroup of $FI_1$ generated by $\{(-n, 0, 0) \colon n \in \mathbb{N} \}$ and $u := (-1, 2, 3)$. 
Note that the idempotents  $\{(-n, 0, 0) \colon n \in \mathbb{N} \}$ appear on the left edge of the diagram shown in Figure \ref{fig:new_diag_1}.
Since the third component of $u$ is $3>0$, it follows that no product involving $u$ or $u^{-1}$ can be equal to any
$(-n, 0, 0) $. Therefore, $S$ is indeed not finitely generated.
\end{example}

Our main aim in this section is to show that every inverse subsemigroup of $FI_1$ is finitely generated modulo its semilattice of idempotents.
The main step in doing this is
to consider an inverse subsemigroup of $FI_1$ which is generated by non-idempotent elements. 
Our first goal will be to show finite generation and presentability for such an inverse subsemigroup. We start this by introducing some results from \cite{Oliveira2005}. Let $S$ be an inverse subsemigroup of $FI_1$ generated by non-idempotents. Define
\[
\begin{aligned}
    a_{\text{min}} &:= \min\{a\in \mathbb{N}_0 \colon (-a, x, b) \in S \text{ for some } x \in \mathbb{Z}, b \in \mathbb{N}_0 \},\\
    b_{\text{min}} &:= \min\{b\in \mathbb{N}_0 \colon  (-a, x, b) \in S \text{ for some } x \in \mathbb{Z}, a \in \mathbb{N}_0 \}.
\end{aligned}
\]
In other words, $a_{\text{min}}$ and $b_{\text{min}}$ are the minimum first and last components of elements in~$S$, respectively.

\begin{lemma}[\cite{Oliveira2005}]
\label{lem:O&S}
With the above notation,  the following hold.
\begin{enumerate}[leftmargin=9mm,itemsep=1mm,label=\textup{(\roman*)}]
\item \label{O&S i} There exists $p \in \mathbb{N}_0$ such that any element $u \in S$ is of the form $(-a, qp, b)$ where $q \in \mathbb{Z}$.
\item \label{O&S ii} There exist $\alpha = (-a_{\emph{min}}, p, b_{\alpha})$ and $\beta = (-a_{\beta}, -p, b_{\emph{min}})$ in $S$.
\item \label{O&S iii} Any element $u \in S$ is of the form $(-(a_{\emph{min}}+n), qp, b_{\emph{min}}+m)$ and $-n \leq qp \leq m$.
\item \label{O&S iv} Further to \ref{O&S iii}, suppose $n \equiv r \pmod p$ and $m \equiv s \pmod p$ where $0 \leq r,s < p$. So, there exist $x, y \in \mathbb{N}_0$ such that $u = (-(a_{\emph{min}}+xp+r), qp, b_{\emph{min}}+yp+s)$. Then, $-x \leq q \leq y$.
\end{enumerate}
\end{lemma}

\begin{remark}
\begin{enumerate}[leftmargin=0.1mm,itemindent=0.8cm,labelwidth=\itemindent,labelsep=0cm,align=left,label={(\arabic*)}]
\item 
The definitions and notations of $S$, $a_{\text{min}}$ and $b_{\text{min}}$ are slightly different in \cite{Oliveira2005}. The inverse subsemigroups in \cite{Oliveira2005} are finitely generated and generating sets may contain idempotent elements of $FI_1$. Nonetheless, $a_{\text{min}}$ and $b_{\text{min}}$ in \cite{Oliveira2005} are the minimum first and last components of non-idempotent elements of a finitely generated inverse subsemigroup of $FI_1$. Removing `finitely generated' in the definitions in \cite{Oliveira2005} gives our definitions and this does not change the value of $a_{\text{min}}$ and $b_{\text{min}}$.
\item 
In \cite{Oliveira2005} it is assumed that $S$ is a finitely generated inverse subsemigroup of $FI_1$ for the proofs of \ref{O&S i}-\ref{O&S iv} of \cref{lem:O&S}. However, the proofs do not use the assumption. Hence, \ref{O&S i}-\ref{O&S iv} hold for a general inverse subsemigroup of $FI_1$ and, in particular, for an inverse subsemigroup of $FI_1$ generated by non-idempotent elements.
\end{enumerate}
\end{remark}

Recall from \cref{sec:prelim} that the idempotents of $FI_1$ are the elements 
$(-a,0,b)$, with $(a,b)\in(\mathbb{N}_0\times\mathbb{N}_0)\setminus\{(0,0)\}$, and that the set of idempotents 
can be identified with $(\mathbb{N}_0\times\mathbb{N}_0)\setminus\{(0,0)\}$.
We say that a subset $A\subseteq \mathbb{N}_0\times \mathbb{N}_0$ is \textit{periodic} if there exists $p\in \mathbb{N}$ such that 
$(a, b)\in A$ implies $(a+p, b), (a, b+p)\in A$. We call the natural number $p$ a \textit{periodicity} of $A$. We sometimes say $A$ is a \textit{$p$-periodic} set.
Certainly, if $A$ is $p$-periodic, then it is also $np$-periodic for all $n \in \mathbb{N}$.

 \citet{Reilly1972} showed that every monogenic subsemigroup $\langle u\rangle$ of $FI_1$, where $u$ is a non-idempotent element of $FI_1$, is isomorphic to $FI_1$. 
 In particular, the semilattice $E(\langle u \rangle)$ is isomorphic to $E(FI_1)$. 
 In fact, if $u = (-a, p, b)$, the semilattice $E(\langle u \rangle)$ is a $p$-periodic subset of  $FI_1$.
 An example illustrating this is shown in
 \cref{fig:semilattice_general}; here $u = (-a, 3, b)$ and $E(\langle u\rangle)$ is a $3$-periodic subset of $E(FI_1)$.

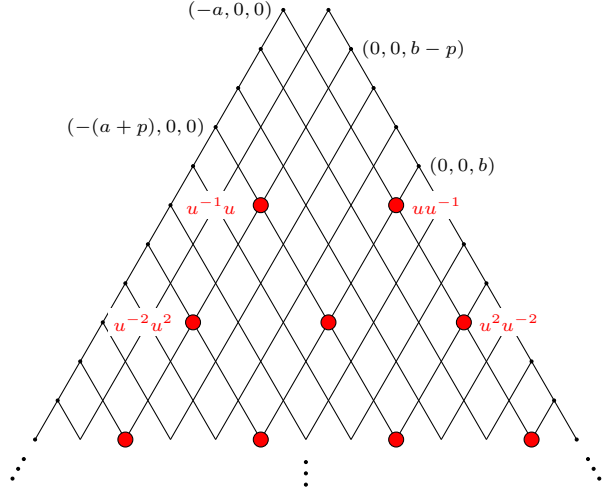
\begin{figure}\centering
\begin{tikzpicture}[scale=0.6]
\node [draw, shape = circle, fill = black, inner sep=-2, scale=0.2] (f1) at (0.5,8) {};
\node[draw, shape=circle, fill=black, inner sep=-2, scale=0.2] (e1) at (-0.5,8) {}; 
\foreach \i in {2, ..., 12}{
    \node[draw, shape=circle, fill=black, inner sep=-2, scale=0.2] (f\i) at ($(f1)+(-60:\i-1)$) {};
    \node[draw, shape=circle, fill=black, inner sep=-2, scale=0.2] (e\i) at ($(e1)+(-120:\i-1)$) {};
}

\draw (f1)--(f12);
\draw (e1)--(e12);

\foreach \i in {1, ..., 11} {
    \draw (f\i) -- ($(e1)+(-120:11)+(0:\i)$);
    \draw (e\i) -- ($(f1)+(-60:11)-(0:\i)$);
}

\node[draw, shape = circle, fill=black, label=left:{\tiny $(-a,0,0)$}, inner sep=-1, scale=0.08] (ea) at (e1) {};
\node[draw, shape = circle, fill=black, label=left:{\tiny $(-(a+p),0,0)$}, inner sep=-1, scale=0.08] (eap) at (e4) {};
\node[draw, shape = circle, fill=black, label=right:{\tiny $(0,0,b-p)$}, inner sep=-1, scale=0.08] (fbp) at (f2) {};
\node[draw, shape = circle, fill=black, label=right:{\tiny $(0,0,b)$}, inner sep=-1, scale=0.08] (fb) at (f5) {};

\node [draw, shape = circle, fill = red, inner sep=-2] (m1) at ($(f6)-(0:1)$) {};
\node [red, fill=white,inner sep=2pt] at ($(m1)+(9mm,0mm)$) {\tiny $uu^{-1}$};

\node [draw, shape = circle, fill = red, inner sep=-2] (m'1) at ($(e6)+(0:2)$) {};
\node [red, fill=white, inner sep=2pt] at ($(m'1)+(-11mm,0mm)$) {\tiny $u^{-1}u$};

\node [draw, shape = circle, fill = red, inner sep=-2] (m2) at ($(f9)-(0:1)$) {};
\node [red, fill=white,inner sep=2pt] at ($(m2)+(10mm,0mm)$) {\tiny $u^2u^{-2}$};

\node [draw, shape = circle, fill = red, inner sep=-2] (m'2) at ($(e9)+(0:2)$) {};
\node [red, fill=white, inner sep=2pt] at ($(m'2)+(-11mm,0mm)$) {\tiny $u^{-2}u^2$};

\node [draw, shape = circle, fill = red, inner sep=-2] at ($(f9)-(0:4)$) {};


\node [draw, shape = circle, fill = red, inner sep=-2] at ($(f12)-(0:10)$) {};
\node [draw, shape = circle, fill = red, inner sep=-2] at ($(f12)-(0:7)$) {};
\node [draw, shape = circle, fill = red, inner sep=-2] at ($(f12)-(0:4)$) {};
\node [draw, shape = circle, fill = red, inner sep=-2] at ($(f12)-(0:1)$) {};
\node [draw, shape = circle, fill = black, inner sep=-2, scale=0.2] at ($(e12)+(-120:0.5)$) {};
\node [draw, shape = circle, fill = black, inner sep=-2, scale=0.2] at ($(e12)+(-120:0.75)$) {};
\node [draw, shape = circle, fill = black, inner sep=-2, scale=0.2] at ($(e12)+(-120:1)$) {};

\node [draw, shape = circle, fill = black, inner sep=-2, scale=0.2] at ($(e12)!0.5!(f12) + (-90:0.5)$) {};
\node [draw, shape = circle, fill = black, inner sep=-2, scale=0.2] at ($(e12)!0.5!(f12) + (-90:0.75)$) {};
\node [draw, shape = circle, fill = black, inner sep=-2, scale=0.2] at ($(e12)!0.5!(f12) + (-90:1)$) {};

\node [draw, shape = circle, fill = black, inner sep=-2, scale=0.2] at ($(f12)+(-60:0.5)$) {};
\node [draw, shape = circle, fill = black, inner sep=-2, scale=0.2] at ($(f12)+(-60:0.75)$) {};
\node [draw, shape = circle, fill = black, inner sep=-2, scale=0.2] at ($(f12)+(-60:1)$) {};

\end{tikzpicture}
\caption{Semilattice of idempotents of $FI_1$. The red dots are idempotents of $\langle u \rangle$ where $u = (-a, 3, b) \in FI_1$. 
}
\label{fig:semilattice_general}
\end{figure}

We now proceed to prove our first main result:

\begin{namedtheorem}[Theorem A]
\label{thm:new_S_bar}
If $S$ is an inverse subsemigroup of the monogenic free inverse semigroup generated by non-idempotent elements, then $S$ is finitely generated, and hence finitely presented.
\end{namedtheorem}

\begin{proof}
We will show that $S$ is finitely generated. Finite presentability will then follow by Theorem~\ref{thm:o&s}.

By assumption, $S=\langle T\rangle$, where $T:=S\setminus E(S)$ is the set of all non-idempotent elements of $S$.
 Note that $T$ is always infinite.
  
\begin{claim}\label{S_bar_claim_1}
The semilattice $E(S)$ is a finite union of semilattices of idempotents of monogenic inverse subsemigroups of $S$, i.e. there is a finite set $T_1\subseteq T$
such that 
\[
E(S) = \bigcup_{u\in T_1} E(\langle u\rangle).
\]
\end{claim}

\begin{proof}
First, we show that for each idempotent $e\in E(S)$ there exists a non-idempotent element $t \in S$ which is $\mathscr{R}$-related to $e$. 
Indeed, if $e = a_1\dots a_l$ with $a_i \in T\cup T^{-1}$, define $t:=ea_1$. Clearly, $t$ is non-idempotent. A straightforward calculation shows that $tt^{-1} = e$, and hence $t \mathscr{R} e$. Therefore, $E(S)$ is an infinite union of semilattices of monogenic inverse subsemigroups of $FI_1$:
\[
E(S) = \bigcup\limits_{t\in T}E(\langle t \rangle). 
\]

Since $E(FI_1)$ and $(\mathbb{N}_0\times\mathbb{N}_0)\setminus\{(0,0)\}$ are isomorphic as ordered sets,
it follows that $E(S)$ has only finitely many maximal elements
$e_{1}, \dots, e_{k}$. 
For each $i \in \{1, \dots, k\}$,  let $t_i := (-a_i, p_i, b_i)\in S$ be a non-idempotent element which is $\mathscr{R}$-related to $e_{i}$. 
Each $E(\langle t_i \rangle)$ is a periodic subset with periodicity $\vert p_i \vert$. Hence the union
\[
X := \bigcup\limits_{i=1}^{k}E(\langle t_i \rangle)
\]
is $q$-periodic with $q = \vert p_1\dots p_k \vert$.
In particular, $X$ contains all the elements of the form $(-(a_i+xq),0, b_i+yq)$ for $i \in \{1, \dots, k\}$ and $x,y\in \mathbb{N}_0$.

Now consider an idempotent $e$ in $E(S)\setminus X$. Since $e_{1}, \dots, e_{k}$ are maximal in $E(S)$, we have
$e=(-(a_i+\alpha), 0,b_i+\beta)$ for some $i \in \{1, \dots, k\}$ and $\alpha, \beta \in \mathbb{N}_0$. 
Dividing $\alpha$ and $\beta$ by $q$, this yields
$e=(-(a_i + xq + r), 0, b_i + yq + s)$ where $i \in \{1, \dots, k\}$, $x, y\in \mathbb{N}_0$ and $r,s\in \{0, 1, \dots, q-1\}$. 
The assumption that $e\in E(S)\setminus X$ simply means that we cannot have $r=s=0$.
Now let $P\subseteq \{1,\dots,k\}\times\{0,\dots,q-1\}\times\{0,\dots,q-1\}$ be the (finite) set of all $(i,r,s)$ that arise in this way.

For each $(i,r,s)\in P$, let $E_{i,r,s}$ be the finite set of maximal elements in the set of all idempotents of $S$ of the form $(-(a_i+xq+r), 0,b_i+yq+s)$.
For each $e\in E_{i,r,s}$ let $u_{i,r,s,e}:=et_i$.
Recalling that $t_i = (-a_i, p_i, b_i)$ we see that $u_{i,r,s,e}$ is a non-idempotent element of $S$ that is $\mathscr{R}$-related to $e$.
Define
\[
Y:= \bigcup_{(i,r,s)\in P} \:\bigcup_{e\in E_{i,r,s}} E(\langle u_{i,r,s,e}\rangle),
\]
and notice that this is again a finite union of semilattices of idempotents of monogenic semigroups, and that it is $q$-periodic.

We claim that $E(S) = X \cup Y$. 
Indeed, any idempotent $e \in E(S)$ is of the form \mbox{$(-(a_i+xq+r), 0,b_i+yq+s)$} for some $i \in \{1, \dots, k\}$, $x, y \in \mathbb{N}_0$, and $(r,s) \in \{0, 1, \dots, q-1\}^{2}$. 
If $(r, s) = (0, 0)$, then $e \in X$. 
Otherwise, $e\leq e'\in E_{i,r,s}$, and hence $e\in E(\langle u_{i,r,s,e'}\rangle)\subseteq Y$.
This proves \cref{S_bar_claim_1}.
\end{proof}

For the rest of the proof of \nameref{thm:new_S_bar}, 
we will fix a finite set $T_1\subseteq T$ guaranteed by \cref{S_bar_claim_1}. 
We aim to show that there exists another finite set $T_2$ of non-idempotent elements of $T$, such that $S = \langle T_1, T_2 \rangle$. 
Recall $a_{\text{min}}, b_{\text{min}}, p$ and the elements
\[
\alpha = (-a_{\text{min}}, p, b_{\alpha}), \;\; \beta = (-a_{\beta}, -p, b_{\text{min}})
\]
from \cref{lem:O&S}. Define 
\[
N := a_{\beta}+b_{\alpha}.
\]


Let us consider an $\mathscr{R}$-class $R_{a,b}$ of $FI_1$ which intersects $S$, and  such that $a+b>N$. 
We know that
$R_{a,b}$ consists of all $(-a,n,b)$ with $-a\leq n\leq b$, and that the second component of any triple in $S$ is divisible by $p$.
Therefore $\{ (-a,qp,b)\colon -a\leq qp\leq b\}\subseteq R_{a,b}\cap S$.
We aim to show that the two sets are in fact equal, and that all these elements can be obtained just from $\epsilon:=(-a,0,b)$, $\alpha$ and $\beta$.

Following \cref{lem:O&S}, write $a$ and $b$  as $a = a_{\text{min}}+xp+r$ and $b = b_{\text{min}}+yp+s$ where $x,y \in \mathbb{N}_0$, $0 \leq r, s <p$. Then any element $(-a, qp, b) \in R_{a,b}\cap S$ has $-x \leq q \leq y$. 
The  following two claims 
show how to obtain the elements $(-a,(q\pm 1)p,b)$ using $u:=(-a,qp,b)$, $\alpha$ and $\beta$.

\begin{claim}\label{new_S_bar_claim_2}
If $-(x-1)p \leq qp \leq b$ then $(-a, (q-1)p, b)\in \langle u,\alpha,\beta\rangle$.
\end{claim}

\begin{proof}
When $-(x-1)p \leq qp \leq b-b_{\alpha} $ we have
\begin{align*}
    u\alpha^{-1} &= (-a, qp, b)(-(a_{\text{min}}+p), -p, b_{\alpha}-p)\\
    &=(-\max(a, a_{\text{min}}+p-qp), (q-1)p, \max(b, b_{\alpha}-p+qp))\\
    &=(-a, (q-1)p, b)\in R_{a,b},
\end{align*}
since $a=a_{\text{min}}+xp+r$;
whereas when $b-b_{\alpha} < qp \leq b$ we have:
\begin{align*}
    u\beta &= (-a, qp, b)(-a_{\beta}, -p, b_{\text{min}})\\
    &=(-\max(a, a_{\beta}-qp), (q-1)p, \max(b, b_{\text{min}}+qp))\\
    &=(-a, (q-1)p, b) \in R_{a,b}.
\end{align*}
The last equality follows from $a+b > a_{\beta}+b_{\alpha}$, which gives $qp > b-b_{\alpha} > -a+a_{\beta}$, and $b = b_{\text{min}}+yp+s$ and $q \leq y$. This completes the proof of Claim \ref{new_S_bar_claim_2}.
\end{proof}

\begin{claim}\label{new_S_bar_claim_3}
If $-a \leq qp \leq (y-1)p$ then $(-a, (q+1)p, b)\in \langle u,\alpha,\beta\rangle$.
\end{claim}

\begin{proof}
The proof is analogous to the that of \cref{new_S_bar_claim_2}:
if $-a + a_{\beta} \leq qp \leq (y-1)p$ then $u\beta^{-1} = (-a, (q+1)p, b)$, and if $-a \leq qp < -a+a_{\beta}$ then $u\alpha = (-a, (q+1)p, b)$.
\end{proof}

\begin{claim}
\label{cl:RabS}
$R_{a,b}\cap S=\{ (-a,qp,b)\::\: -a\leq qp\leq b\} \subseteq \langle \epsilon,\alpha,\beta\rangle$.
\end{claim}

\begin{proof}
Since $R_{a,b}\cap S\neq\emptyset$, certainly $\epsilon\in R_{a,b}\cap S$.
Successively applying Claims \ref{new_S_bar_claim_2}, \ref{new_S_bar_claim_3} yields 
$\{ (-a,qp,b)\::\: -a\leq qp\leq b\} \subseteq \langle \epsilon,\alpha,\beta\rangle\subseteq S$, and the claim follows.
\end{proof}

We can now complete the proof of \nameref{thm:new_S_bar} as follows. Recall that each $\mathscr{D}$-class in $FI_1$ is finite. Let 
$T_{2} := \bigl(\bigcup_{i=1}^{N}D_i\bigr)\cap S$, a finite set.
Note that $\alpha, \beta \in T_2$ as $a_{\text{min}}+b_{\alpha}, a_{\beta}+b_{\text{min}} \leq N$. 
We claim that $S = \langle T_1, T_2 \rangle$. 
To see this note that $E(S) \subseteq \langle T_1, T_2\rangle$ by \cref{S_bar_claim_1}. 
Now, consider an arbitrary $R_{a,b}$ which intersects $S$.
If $a+b\leq N$ then $R_{a,b}\cap S\subseteq D_{a+b} \subseteq T_2\subseteq \langle T_1,T_2\rangle$.
If $a+b>N$, then Claim \ref{cl:RabS} gives $R_{a,b}\cap S\subseteq \langle E(S),\alpha,\beta\rangle\subseteq \langle T_1,T_2\rangle$.
Hence $S$ is finitely generated, and the proof of the theorem is complete.
\end{proof}

We can immediately derive:

\begin{namedtheorem}[Corollary B]
\label{cor:gen}
Let $S$ be an inverse subsemigroup of the monogenic free inverse semigroup. Then, $S$ is finitely generated modulo the semilattice of its idempotents.
\end{namedtheorem}

\begin{proof}
Let $\overline{S}$ be the inverse subsemigroup of $S$ that is generated by all non-idempotent elements of $S$. 
Then $S = \langle E(S), \overline{S} \rangle$, and the corollary follows from 
\nameref{thm:new_S_bar}.
\end{proof}

We finish this section with an observation on the subsemigroup $\overline{S}$ generated by non-idempotent elements:

\begin{proposition}
\label{prop:ideal}
Let $S$ be an inverse subsemigroup of $FI_1$. The inverse subsemigroup $\overline{S}$ of $S$ generated by all non-idempotent elements of $S$ is an ideal.
\end{proposition}

\begin{proof}
Let $x\in S$ and $y\in \overline{S}$, and we want to prove that $xy,yx\in\overline{S}$. 
We will prove that $xy\in \overline{S}$, and the claim for $yx$ follows by symmetry. If $x\in \overline{S}$ then $xy\in\overline{S}$; so let us suppose that $x\in S\setminus\overline{S}$, which certainly means that $x$ is an idempotent. Now, if $y$ is a non-idempotent element, then it is easy to see that $xy$ is also non-idempotent, and hence $xy\in\overline{S}$. Finally, if $y$ is an idempotent, then $y=zz^{-1}$ for some non-idempotent element $z$ by Claim~\ref{S_bar_claim_1} in the proof of \nameref{thm:new_S_bar}, so that
$xy=(xz)z^{-1}\in \overline{S}$.
\end{proof}


\section{Presentations}
\label{sec:new_present}

In this section, we will obtain presentations for a general inverse subsemigroup of $FI_1$ using the generating set result, \nameref{cor:gen}. 

Recall that given an inverse semigroup $S$ and an inverse subsemigroup 
$T $ of $S$ we say that $S$ is finitely presented modulo $T$ if $S = \Inv \langle X, Y \mid R, Q \rangle$ where
$T= \Inv \langle X \mid R \rangle$ and $Y$ and $Q$ are finite (\cref{def:gen_mod}). Our guiding question for this section is whether every inverse subsemigroup of $FI_1$ is finitely presented modulo the semilattice of its idempotents. To treat this question, we first give a `nice' infinite presentation:

\begin{theorem}
\label{thm:new_amlagam}
Let $S$ be a non-semilattice inverse subsemigroup of $FI_1$. 
Let $\Inv \langle X_E \mid R \rangle$ be any presentation for the semilattice $E(S)$ of the idempotents of $S$ in terms of the generating set
$X_E := \{x_e \colon e\in E(S)\}$ in one-one correspondence with $E(S)$. 
Also let $\Inv \langle Y \mid Q \rangle$ be a presentation for the subsemigoup $\overline{S}$ of $S$ generated by all the non-idempotent elements of $S$.
For each $e \in E(S) \cap \overline{S}$, let $w_e\in (Y\cup Y^{-1})^+$ be a word that represents $e$ in $\overline{S}$, and let 
\begin{align*}
    P := \{w_e = x_e \colon e \in E(S)\cap \overline{S}\}.
\end{align*}
Then, $S$ is defined by 
\begin{align}
    \Inv \langle X_E, Y \mid R, Q, P \rangle. \label{eq:new_amalgam_1}
\end{align}
\end{theorem}

\begin{remark}
The reader may think of $\Inv\langle X_E\:|\: R\rangle$ as the Cayley table of $E(S)$, i.e. of $R$ as consisting of the relations $x_ex_f=x_{ef}$, $e,f\in E(S)$. Also, by \nameref{thm:new_S_bar}, the presentation $\Inv\langle Y\:|\: Q\rangle$ for $\overline{S}$ can be taken to be finite.
\end{remark}

\begin{proof}[Proof of \cref{thm:new_amlagam}]
It is clear that $S$ is generated by a set in one to one correspondence with $X_E\cup Y$ and that it satisfies the relations $R \cup Q \cup P$. Therefore, $S$ is a homomorphic image of the inverse semigroup defined by the presentation (\ref{eq:new_amalgam_1}). 
It remains to prove that for words $u_1$ and $u_2$ over $X_E \cup  Y $, if the equality $u_1 = u_2$ holds in $S$, then it can be deduced from $R \cup Q \cup P$.

In order to prove this, we claim that for every word $u\in (X_E\cup X_E^{-1}\cup Y\cup Y^{-1})^+$ there exists a word $\overline{u}\in (X_E\cup X_E^{-1})^+\cup (Y\cup Y^{-1})^+$ such that $u=\overline{u}$ is a consequence of $R \cup Q \cup P$.
The assertion is non-obvious only if $u$ contains letters from both $X_E\cup X_E^{-1}$ and $Y\cup Y^{-1}$.
In that case $u$ contains a subword of the form $x_e^\epsilon y^\delta$ or $y^{-\delta} x_e^{-\epsilon}$ with $e\in E$, $y\in Y$, $\epsilon,\delta\in\{-1,1\}$.
The element $y^\delta y^{-\delta}$ represents an idempotent $f$ from $\overline{S}$, and using \cref{prop:ideal} we have $g:=ef\in E(S)\cap \overline{S}$. Now we have the deduction
\[
x_e^\epsilon y^\delta = x_e^\epsilon y^\delta y^{-\delta} y^\delta \stackrel{Q}{=} x_e^\epsilon w_f y^{\delta}
\stackrel{P}{=} x_e^\epsilon x_f y^\delta\stackrel{R}{=}
x_g y^{\delta} \stackrel{P}{=} w_g y^\delta\in (Y\cup Y^{-1})^+.
\]
Taking inverses we see that also $y^{-\delta}x^{-\epsilon}$ can be transformed into a word over $Y$ using $R \cup Q \cup P$. Using this repeatedly, one can  eliminate all the occurrences of letters from $X_E\cup X_E^{-1}$ in $u$, and end up with a word $\overline{u}$ solely over $Y$.

Now, returning to our relation $u_1=u_2$ that holds in $S$, consider $\overline{u}_1$, $\overline{u}_2$.
If both $\overline{u}_1$ and $\overline{u}_2$ are words over $X_E$ (resp.\ $Y$) then the relation
$\overline{u}_1=\overline{u}_2$ is a consequence of $R$ (resp.\ $Q$).
Otherwise, we may assume without loss that $\overline{u}_1\in (X_E\cup X_E^{-1})^+$ and $\overline{u}_2\in (Y\cup Y^{-1})^+$. In this case they both represent an idempotent $e\in E(S)\cap \overline{S}$, and we have
\[
\overline{u}_1\stackrel{R}{=} x_e\stackrel{P}{=} w_e\stackrel{Q}{=} \overline{u}_2.
\]
In all the cases we have deduced $\overline{u}_1=\overline{u}_2$ as a consequence of $R \cup Q \cup P$,
and then $u_1=\overline{u}_1=\overline{u}_2=u_2$ is a derivation of $u_1=u_2$, as required.
\end{proof}

\begin{remark}
From the presentation (\ref{eq:new_amalgam_1}), one may recognise that $S$ is the free product of $E(S)$ and $\overline{S}$ with amalgamation $E(S)\cap \overline{S}$. For semigroup amalgams in general, see Chapter 8 of \cite{Howie}. For inverse semigroup amalgamated free products, see \cite{Bennett} and \cite{StephenAmalgam}.
\end{remark}

The presentation established in \cref{thm:new_amlagam} is infinite. In fact, intuitively one could say it is infinite for a number of different `reasons':
\begin{itemize}[leftmargin=8mm]
\item
The set of relations $R$ defining the semilattice $E(S)$ is infinite -- this is of necessity, as $E(S)$ is infinite.
\item
The set $P$ `linking' $E(S)$ and $\overline{S}$ is infinite, and this is the bar to concluding that $S$ is finitely presented modulo $E(S)$.
\item
In fact, the set $P$ features infinitely many elements of $\overline{S}$.
\end{itemize}

It will turn out that $S$ is not finitely presented modulo $E(S)$. But we are able to establish another `nice' presentation, which is based on the conjugation action of $S$ on $E(S)$, in which only finitely many words over the non-idempotent generators appear.

\begin{corollary}
\label{cor:new_conj}
Let $S$ be a non-semilattice inverse subsemigroup of $FI_1$. Suppose $A$ is a finite set of non-idempotents such that $S = \langle E(S), A \rangle$. Let $X_E := \{x_e\colon e\in E(S)\}$ and $Y_A := \{y_a \colon a \in A \}$ be sets in one-one correspondence with $E(S)$ and $A$, respectively. Suppose $E(S)$ is defined by $\Inv \langle X_E \mid R \rangle$. Define
\begin{align*}
    C := \{y_a^{-\epsilon} x_ey_a^\epsilon = x_{a^{-\epsilon}ea^\epsilon}  \colon a\in A,\, e\in E(S),\, \epsilon=\pm 1 \}.
\end{align*}
Then, there is a finite set $T$ of relations such that $S$ is defined by
\begin{align*}
    \Inv \langle X_E, Y_A \mid R, C, T \rangle.
\end{align*}
\end{corollary}

\begin{proof}
For $a \in A$ and $\epsilon=\pm 1$ both $y_{a}^\epsilon y_{a}^{-\epsilon}$ and $x_{a^\epsilon a^{-\epsilon}}$ correspond to the idempotent $a^\epsilon a^{-\epsilon}$. 
Define:
\begin{align*}
    T_1 := \{y_a^\epsilon y_a^{-\epsilon} = x_{a^\epsilon a^{-\epsilon}}\::\:  a\in A,\ \epsilon=\pm 1\}.
\end{align*}
By \nameref{thm:new_S_bar} the inverse subsemigroup $\overline{S}$ generated by all non-idempotent elements of $S$ is finitely presented.
Since $A$ consists of non-idempotent elements, 
it can be extended to a finite generating set for $\overline{S}$,
and then $\overline{S}$ can be defined by a finite presentation over the resulting generating set.
In other words, $\overline{S}$ is defined by a finite presentation 
$\Inv \langle Y \mid Q \rangle$ where $Y_A \subseteq Y$. 
We write $Y_1$ for the set $Y\setminus Y_A$. Each $y_1 \in Y_1$ corresponds to an element in $S$. This element can be written as a product of generators $E(S)\cup A \cup A^{-1}$, and so there is a word $w_{y_1}$ over 
$X_E\cup Y_A$ which corresponds to the element. 
Let
\begin{align*}
    T_2 := \{ y_1 = w_{y_1} \colon y_1\in Y_1\}.
\end{align*}

We claim that $S$ is defined by 
\begin{align}
    \Inv \langle X_E, Y\mid R, Q, C, T_3\rangle \label{eq:new_conj_1}
\end{align}
where $T_3 := T_1 \cup T_2$. 
Once we prove the claim, the proof of the corollary is easy to complete. Indeed, the generators in $Y_1$ are redundant via relations $T_2$. Thus we may eliminate those generators, which results in 
removing the relations $T_2$, and substituting $w_{y_1}$ for $y_1$ in all other relations.
But, in fact $y_1$ can only appear in relations $Q$, which after the substitution become a new finite set of relations $Q'$.
 Then, defining $T:=Q' \cup T_1$ completes the proof of the corollary.

We now turn to the proof of the claim.
Recall the set $P$ of relations in \cref{thm:new_amlagam}:
\begin{align*}
    P = \{w_e = x_e \colon e \in E(S)\cap \overline{S}\}
\end{align*}
where each $w_e$ is an arbitrary word over $Y$ which represents the idempotent $e$. 
Thus, we may take $w_e$ to have the form $w_e\equiv w_1w_1^{-1}$, $w_1\in (Y\cup Y^{-1})^+$.
By \cref{thm:new_amlagam}, $S$ is defined by
\begin{align*}
    \Inv \langle X_E, Y \mid R, Q, P \rangle.
\end{align*}
It is clear that $S$ satisfies the relations in $C \cup T_3$, and so the set $R\cup Q \cup P$ of relations implies $C \cup T_3$. We need to prove that the set $R \cup Q\cup C\cup T_3$ of relations implies $P$. 
To this end consider an arbitrary $e \in E(S)\cap \overline{S}$.

First, we  use relations from $T_2$ to eliminate generators from $Y_1$ in $w_1$, yielding a word
$w_2\in (X_E\cup X_E^{-1}\cup Y_A\cup Y_A^{-1})^+$, which must contain at least one occurrence of a letter from $Y_A$.
Thus $w_e=w_2 w_2^{-1}$ as a consequence of $R \cup Q\cup C\cup T_3$.

Now consider the last occurrence of a letter from $Y_A$ in $w_2$; in other words suppose
$w_2 \equiv uy_a^\epsilon v$, where $y_a\in Y_A$, $\epsilon=\pm 1$, and $v\in (X_E\cup X_E^{-1})^*$.
If the word $v$ is empty, denoting by $g$  the idempotent represented by $a^\epsilon a^{-\epsilon}$, we have
\[
w_e=w_2 w_2^{-1}\equiv u y_a^\epsilon y_a^{-\epsilon} u^{-1}
\stackrel{T_1}{=} u x_g u^{-1}
\stackrel{R}{=} u x_g x_g^{-1} u^{-1}.
\]
When $v$ is non-empty, it is equal to some $x_f$ as a consequence of $R$; 
denoting by $g$ the idempotent represented by $a^\epsilon fa^{-\epsilon}$, we have:
\[
w_e=w_2w_2^{-1} \stackrel{R}{=} u y_a^\epsilon x_f x_f^{-1} y_a^{-\epsilon} u^{-1}
\stackrel{R}{=}u y_a^\epsilon x_f  y_a^{-\epsilon} u^{-1}
\stackrel{C}{=} u x_g   u^{-1}
\stackrel{R}{=} u x_g x_g^{-1}   u^{-1}.
\]
Denoting $u x_g$ by $w_3$, in each case we have proved that
$w_e=w_2w_2^{-1}=w_3 w_3^{-1}$ as a consequence of $R \cup Q\cup C\cup T_3$, and the number of occurrences of letters from $Y_A$ in $w_3$ is one less than in $w_2$.
Continuing in this way we can successively eliminate all the occurrences of letters from $Y_A$,
and therefore deduce $w_e=x_e$ as a consequence of $R \cup Q\cup C\cup T_3$, which is the desired relation from $P$.
This completes the proof of the claim, and hence of the corollary.
\end{proof}

Finally, we return to our main question of whether every inverse subsemigroup of $FI_1$ is finitely presented modulo the semilattice of its idempotents. The following gives the negative answer:

\begin{namedtheorem}[Theorem C]
\label{thm:new_no_fin}
Let $S$ be a non-semilattice inverse subsemigroup of the monogenic free inverse semigroup. Then, $S$ is not finitely presented modulo the semilattice of its idempotents.    
\end{namedtheorem}

\begin{proof}
Aiming for a contradiction, suppose that $S$ is finitely presented modulo $E(S)$.
Hence, $S$ can be defined by a presentation
\begin{align}
    \Inv \langle X_E, Y \mid R, T \rangle,    \label{eq:new_no_fin_1}
\end{align}
where $X_E= \{x_e \colon e\in E(S)\}$ is a (necessarily infinite) set in one to one correspondence with $E(S)$,
$\Inv\langle X_E\mid R\rangle$ is a presentation for $E$, and the sets $Y$ and $T$ are finite.
For the purposes of this proof, we will take $R$ to be the set of relations arising from the Cayley table of $E$, i.e. $x_ex_f=x_{ef}$, $e,f\in E$.

Let $\overline{S}$ be the inverse subsemigroup of $S$ generated by all non-idempotent elements of $S$. 
Using \nameref{thm:new_S_bar}, we may without loss assume that 
the set $Y$ corresponds to a finite set of non-idempotent elements in $S$, which generates $\overline{S}$.
Furthermore,  \nameref{thm:new_S_bar} also gives that $\overline{S}$ is defined by a finite presentation $\Inv \langle Y \mid Q\rangle$. 
Now \cref{thm:new_amlagam} implies that $S$ is defined by
\begin{align}
    \Inv \langle X_E, Y \mid R, Q, P\rangle, 
    \quad \text{with}\quad
    P := \{w_e = x_e \colon e \in E(S)\cap \overline{S}\},
    \label{eq:new_no_fin_2}
\end{align}
where $w_e\in (Y\cup Y^{-1})^+$ represents $e$.
Since both presentations  (\ref{eq:new_no_fin_1}) and $(\ref{eq:new_no_fin_2}$) define $S$, we can deduce $T$ from $Q\cup P$. 
Since $T$ is finite,
only finitely many relations from $Q\cup P$ are used in such a deduction. 
Hence there exists a finite subset $P_1$ of $P$ such that $S$ is defined by a presentation
\begin{align}
    \Inv \langle X_E, Y\mid R, Q, P_1\rangle,
    \quad \text{with} \quad P_1 = \{w_{e_1} = x_{e_1}, \dots, w_{e_k} = x_{e_k}\}.
 \label{eq:new_no_fin_3}
\end{align}

Consider the element $f:=e_1\dots e_k\in E(S)$. 
Then the relation $x_{e_1}\dots x_{e_k} = x_f$ is a consequence of $R$. Also, we have $x_f = w_f$ as a consequence of $R\cup Q \cup P_1$. 
Since $w_f$ is a word over $Y$, 
there is a letter $y \in Y$ and $\epsilon=\pm 1$ such that $w_f \leq y^\epsilon y^{-\epsilon}$. 
Recall that the semilattice of idempotents of $\langle y \rangle$ consists of the elements of the form $y^{-i}y^iy^jy^{-j}$ (see \cref{fig:semilattice_general}). 
In particular, there exist $p,q\in\mathbb{N}_0$ such that the idempotent $g$ represented by
$y^{-p}y^py^qy^{-q}$ is strictly smaller than $f$. 
The relation $y^{-p}y^py^qy^{-q}=x_g$ holds in $S$.
In particular, the word $x_g$ belongs to the language of the 
Sch\"{u}tzenberger graph $\SGamma(y^{-p}y^py^{q}y^{-q})$ with respect to $X_E\cup Y$.

On the other hand, $\SGamma(y^{-p}y^py^{q}y^{-q})$ can be obtained by using Stephen's procedure
with respect to the presentation \eqref{eq:new_no_fin_3}.
We are going to reach a contradiction by showing inductively that at every stage of the procedure the following holds:
\begin{equation}
\label{eq:xlab}
\text{if } x_h \text{ is an edge label then } h\geq f; 
\end{equation}
this will immediately imply that there will never be an edge labelled $x_g$.
This is certainly the case at the start of the procedure, i.e. for the linear graph of $y^{-p}y^py^{q}y^{-q}$ which contains no edges labelled by $x_h$ for any $h$.
Also, it is clear that a determination does not introduce new labels. 
Consider the possibility that an expansion of the graph that satisfies \eqref{eq:xlab} yields a graph that does not.
Notice that the relations from $Q$ contain no symbols from $X_E$, while those in $P_1$ contain only $x_{e_1},\dots,x_{e_k}$, and $e_1,\dots,e_k\geq f$.
Hence, none of these relations can be used to introduce a label $x_h$ with $h<f$.
The remaining relations $R$ have the form $x_{h'} x_{h''}=x_h$, where $h,h',h''\in E$ and $h=h'h''$.
Notice that $h\geq f$ if and only if $h',h''\geq f$, and it follows that this relation does not introduce a label corresponding to an idempotent below $f$ either.
This completes the proof of the claim, and the theorem follows.
\end{proof}


\section{Conclusions and further questions}
\label{sec:con}

To conclude the paper we make a remark which records an interesting relationship between an inverse subsemigroup $S$ of $FI_1$ and the subsemigroup $\overline{S}$ generated by the non-idempotents, and we also pose a few questions for possible further investigation.

\begin{proposition}
\label{prop:final}
Let $S$ be a non-semilattice inverse subsemigroup of $FI_1$, and let $\overline{S}$ be the inverse subsemigroup of $S$ generated by all non-idempotent elements of $S$. Then the following are
equivalent: \textup{(i)} $S$ is finitely generated; \textup{(ii)} $S$ is finitely presented;
\textup{(iii)} the set $S \setminus \overline{S}$ is finite.
\end{proposition}

\begin{proof}
The equivalence of (i) and (ii) is an immediate consequence of Theorem~\ref{thm:o&s}.
Also, (iii)$\Rightarrow$(i) is an easy consequence of the fact that $\overline{S}$ is finitely generated (\nameref{thm:new_S_bar}).
For (i)$\Rightarrow$(iii) suppose that $S$ is finitely generated, say by $A$.
Recall that $\overline{S}$ is an ideal of $S$ (\cref{prop:ideal}).
It follows that $S\setminus\overline{S}\subseteq \langle A\setminus\overline{S}\rangle$.
But $A\setminus\overline{S}$ is a finite set of idempotents, and the subsemigroup generated by it is also finite.
%
%
\end{proof}



One natural direction for further work arising from this paper is to consider `plain' subsemigroups of $FI_1$, and ask about their finite generation and presentability.
\citet{Schein1975} proved that free inverse semigroups, including $FI_1$, are not finitely presented as semigroups. On the other hand, $FI_1$ certainly contains finitely presented subsemigroups, e.g. finite semilattices of idempotents, and lots of copies of the monogenic free semigroup $\mathbb{N}$.
The question of characterising finite presentability of subsemigroups of $FI_1$ is considered in the forthcoming paper \cite{paper2}.

Another natural direction is to investigate inverse subsemigroups of free inverse semigroups of higher rank.
%
%
\citet{Oliveira2005} showed that there is a finitely generated inverse subsemigroup of a free inverse semigroup of rank greater than or equal to $2$ which is not finitely presented. The next question is motivated by this result, which was also noted in~\cite{Oliveira2005}. 

\begin{question}[\cite{Oliveira2005}, Problem 6.2]
\label{question:rank_geq_2}
Is it decidable whether a finitely generated inverse subsemigroup of a free inverse semigroup of rank greater than or equal to $2$ is finitely presented or not?
\end{question}

A related question  is to decide if an inverse subsemigroup of a free inverse semigroup is free. \citet{Reilly1972} gave a necessary and sufficient condition for an inverse subsemigroup of a free inverse semigroup being free inverse, but it is not clear whether this condition translates into a decision algortithm.

\begin{question}
Is it decidable whether a finitely generated inverse subsemigroup of a free inverse semigroup is free inverse or not?
\end{question}


\end{document}